\documentclass[11pt,letterpaper]{amsart}
\usepackage{a4wide,amsmath,amsbsy,amsfonts,amssymb,
stmaryrd,amsthm,mathrsfs,graphicx, amscd, tikz-cd, cancel, supertabular}
\usepackage[normalem]{ulem}
\usetikzlibrary{matrix,arrows,decorations.pathmorphing}
\usepackage{subcaption}

\usepackage{tikz}
\usetikzlibrary{matrix,arrows,decorations.pathmorphing}
\usepackage[all,cmtip]{xy}
\usepackage{qtree}

\usepackage[top=1.2in,bottom=1.2in,left=1.21in,right=1.21in]{geometry}
\usepackage[
	hypertexnames=false,
	hyperindex,
	pagebackref,
	pdftex,
	breaklinks=true,
	bookmarks=false,
	colorlinks,
	linkcolor=blue,
	citecolor=red,
	urlcolor=red,
]{hyperref}
\usepackage{hyperref}
\newtheorem{theorem}{Theorem}[section]
\newtheorem{proposition}[theorem]{Proposition}

\newtheorem{lemma}[theorem]{Lemma}

\theoremstyle{definition}

\newtheorem{definition}[theorem]{Definition}
\newtheorem{example}[theorem]{Example}

\newtheorem{problem}[theorem]{Problem}

\numberwithin{equation}{section}

\theoremstyle{definition}

\newtheorem{remark}[theorem]{Remark}
\newtheorem{remarks}[theorem]{Remarks}

\newcommand{\Ec}{\mathcal{E}}
\newcommand{\Fc}{\mathcal{F}}

\newcommand{\Kc}{\mathcal{K}}
\newcommand{\Tc}{\mathcal{T}}

\newcommand{\Pc}{\mathcal{P}}

\newcommand{\Cb}{\mathbb{C}}

\newcommand{\Fb}{\mathbb{F}}

\newcommand{\Rb}{\mathbb{R}}
\newcommand{\Zb}{\mathbb{Z}}

\newcommand{\beq}{\begin{eqnarray}}
\newcommand{\eeq}{\end{eqnarray}}

\newcommand\ssm{\smallsetminus}

\newcommand{\Fcal}{\mathcal{F}}

\newcommand{\Hc}{\mathcal{H}}

\DeclareMathOperator{\aut}{Aut}
\DeclareMathOperator{\En}{E}

\DeclareMathOperator{\GL}{GL}

\DeclareMathOperator{\HS}{HS}

\DeclareMathOperator{\Id}{Id}

\DeclareMathOperator{\Hom}{Hom}

\DeclareMathOperator{\tr}{Trace}

\DeclareMathOperator{\Gras}{Gr}

\DeclareMathOperator{\Riem}{Riem}

\DeclareMathOperator{\SL}{SL}
\DeclareMathOperator{\SO}{SO}

\DeclareMathOperator{\Teich}{Teich}

\DeclareMathOperator{\la}{\langle}
\DeclareMathOperator{\ra}{\rangle}

\DeclareMathOperator{\Diff}{Diff}

\DeclareMathOperator*{\esssup}{ess\,sup}
\DeclareMathOperator{\TE}{TE}
\DeclareMathOperator{\HE}{HE}

\newcommand{\Dil}{\operatorname{Dil}}
\newcommand{\vol}{\operatorname{vol}}

\begin{document}

\title[Extremal mappings of tori and Teichm\"uller potentials]
{Extremal mappings of tori, Teichm\"uller potentials and symmetric-space distance}

\author{Benson Farb}
\address{Department of Mathematics,
University of Chicago,
Chicago, IL 60637, USA}
\email{farb@math.uchicago.edu}

\author{Eduard Looijenga}
\address{Mathematisch Instituut,
Universiteit  Utrecht,
3508 TA Utrecht, 
Nederland}
\email{e.j.n.looijenga@uu.nl}

\date{\today}

\maketitle

\begin{abstract}
The symmetric space $X_n={\rm SL}(n,\Rb)/{\rm SO}(n)$ can be interpreted as the Teichm\"uller space of marked, unit volume, flat $n$-dimensional tori. It comes with a unique (up to scale) ${\rm SL}(n,\Rb)$-invariant metric $d_{X_n}$.  In 1939 Teichm\"uller gave a modular interpretation of $d_{X_2}$ (the hyperbolic metric) in terms of an extremal mapping problem for quasiconformal dilatation.  Such a modular interpretation for $d_{X_n}$ for $n\geq 3$ has remained unaddressed: the natural candidates - minimal quasiconformal dilatation, Lipschitz constant, or total energy - do not work.  

In this paper we give such a modular interpretation, two in fact. We introduce the {\em total expansion} $\TE(f)\in [0,\infty]$ of a Lipschitz map $f:M\to N$ between Riemannian manifolds, a notion related to the notion of ``$k$-dilatation'' developed by Gromov, Guth and others.  For volume-preserving Lipschitz maps $f:\Tc_0\to\Tc_1$ between $n$-dimensional, flat, unit-volume tori, we prove that $\TE(f)$ is minimized in the homotopy class of $f$ precisely by the affine maps in that class and takes on these the value $d_{X_n}$.  We prove similar results for the \emph{Hilbert-Schmidt expansion} $\HE(f)$, which is a simple integral over $M$ and has more of an $L^2$ flavor.
\end{abstract}

\section{Introduction}
The space $X_n$ of inner products on $\Rb^n$ of determinant $1$  is the symmetric space of $\SL(n,\Rb)$:
\[X_n:=\SL(n,\Rb)/\SO(n)\]
and  thus comes equipped with its $\SL(n,\Rb)$-invariant Riemannian metric, normalized in a (standard) manner, which we shall later recall, see Equation \eqref{eq:standard-metric-Xn}.  We  denote the associated path metric on $X_n$ by $d_{X_n}$. For example, $X_2$ is the hyperbolic plane.  The space $X_n$ 
has an interpretation as the {\em Teichm\"uller space} $\Teich(T^n)$ of the $n$-torus $T^n$, defined as the set of equivalence classes of pairs $(\Tc,\Fc)$, 
where $\Tc$ is a flat, unit-volume torus, $\Fc:T^n\xrightarrow{\sim} \Tc$ the class of a  homotopy equivalence (equivalently homology equivalence), and where two such pairs $(\Tc_i,\Fc_i)$ are considered to be equivalent if $\Fc_1\circ \Fc_0^{-1}$ contains an isometry.  Recall that a 
flat $n$-torus $\Tc$ (i.e., a torus endowed  with a flat Riemannian metric) is  isometric to Euclidean $n$-space modulo a lattice. In particular, the  identity component of the isometry group of $\Tc$ 
appears as a group (the ``group of  translations'' of $\Tc$)  acting simply  transitively on $\Tc$.

Given this geometric interpretation of $X_n$, it is a natural problem to find a geometric interpretation 
of the distance $d_{X_n}((\Tc_0,\Fc_0),(\Tc_1,\Fc_1))$ in terms of the solution to an extremal mapping 
problem for the homotopy class $\Fc_1\circ \Fc_0^{-1}:\Tc_0\to \Tc_1$.  When $n=2$, this problem was solved by Teichm\"uller in 1939; see p. 342 of \cite{Teich39}.  He later used the same extremal problem to define a metric on (what we now call) the Teichm\"uller space of Riemann surfaces of genus $g\geq 2$, the study of which is Teichm\"uller theory.  Finding such an interpretation when $n>2$  has, as far as we can tell, remained 
unaddressed.  The natural candidates - minimal quasiconformal dilatation, Lipschitz constant, or total energy - do not give $d_{X_n}$; see Remark 
\ref{remark:dontwork} below.

The main result of this paper solves this problem in two ways, each by minimizing a 
\emph{Teichm\"uller potential}: a real-valued function on the space of Lipschitz maps $f:M\to N$ between closed Riemannian manifolds. One, which we call the {\em total expansion} 
$\TE(f)$,  packages Gromov-Guth's ``$k$-dilatations'' $\Dil_k(f), k\geq 1$ into a single number; 
the other, the \emph{Hilbert-Schmidt expansion} $\HE(f)$,  does something similar  with what we call the \emph{Hilbert-Schmidt energies} $\En_k(f)$. Whereas each $\Dil_k(f)$ involves an essential supremum over $M$ (it is of $L^\infty$-type), the invariant $\En_k(f)$ is an integral over $M$ (it is of $L^2$-type). 

We give the definitions  below in this introduction, but  to get a feel for the  Hilbert-Schmidt expansion: we shall see that it has a pretty interpretation in terms of classical differential geometry: if $g_0$ and $g_1$ are Riemannian metrics on a closed manifold $M$, then the Hilbert-Schmidt expansion of the identity map $f: (M, g_0)\to (M, g_1)$ is the geodesic distance associated with a 
well-known Riemannian metric  on the (infinite-dimensional) space of \emph{all} 
Riemannian metrics on $M$ inducing the standard volume form; see Proposition \ref{prop:WP}.

\subsection{Main theorem} Our main result is the following.

\begin{theorem}[{\bf Main Theorem}]
\label{theorem:newmain}
Let $\Tc_0$ and $\Tc_1$ be flat tori  of the same dimension $n\ge 1$. 
 Assume  that $d:=\vol(\Tc_0)/\vol(\Tc_1)$ is a positive integer and let $\Fc$ be a homotopy class of maps $\Tc_0\to \Tc_1$ of absolute degree $d$.  
 \medskip
 
\begin{enumerate}
\item The functions  $\Dil_k$, $\TE$, $\En_k$ and $\HE$   have a minimum  on the set  of  \emph{volume-preserving} Lipschitz 
maps  $f$ in $\Fc$  (\footnote{By ``volume-preserving'' we mean $|Jac_f(x)|=1$ 
almost everywhere.  Note that Lipschitz maps $f$ are differentiable almost everywhere by Rademacher's Theorem.}) and take that  minimum on the space of affine maps  in $\Fc$. For $\TE$ and $\HE$ this minimum is taken on affine maps \emph{only} (\footnote{It can be shown that this is also true for  $\En_1, \dots, \En_{n-1}$.}).
\bigskip

\item[(2)] Both Teichm\"uller potentials $\TE$ and $\HE$  induce the symmetric space 
distance on $X_n$:  if  $(\Tc_0,\Fc_0)$ and $(\Tc_1,\Fc_1)$ are marked flat unit volume $n$-tori
then 
\begin{equation}\label{eqn:TEmin}
\textstyle d_{X_n}((\Tc_0,\Fc_0),(\Tc_1,\Fc_1))=\inf_{h\in \Fc_1\circ \Fc_0^{-1}}\TE(h)=\inf_{h\in \Fc_1\circ \Fc_0^{-1}}\HE(h)
\end{equation}
where the infimum is taken over all volume-preserving Lipschitz maps $h:\Tc_0\to \Tc_1$ in the homotopy class $\Fc_1\circ \Fc_0^{-1}$.
\end{enumerate}
\end{theorem}

Recall that $f:\Tc_0\to \Tc_1$ \emph{affine} if it is equivariant with respect to a homomorphism on the translation groups determined by $\Fc$.  Note that if a homotopy class $\Fc$ of $f$ 
is specified then there is only one such map up to translation, and that if such a map has absolute degree $d$, then it will be a $d$-sheeted covering.

\begin{remarks}\label{rems:main}
\ 
\begin{enumerate}

\item Note that once Item (1) of Theorem \ref{theorem:newmain} is established, 
Item (2) comes down to the following concrete statement:
Let $g_i$ ($i=0,1$)  be inner products on $\Rb^n$ of determinant $1$, thus defining  points of $X_n$.
Let $\Phi$ stand for the identity map of $\Rb^n/\Zb^n$, but  regarded as a map between flat tori 
$(\Rb^n/\Zb^n, g_0)\to (\Rb^n/\Zb^n, g_1)$. Then 
\[\TE(\Phi)=\HE(\Phi)=d_{X_n}(g_0, g_1).\] This is  in fact what we will prove. 

%\item Combining Parts (1) and (2) of Theorem \ref{theorem:newmain} to marked flat unit volume $n$-tori gives
%\begin{equation}
%\label{eq:easymetric}
%d_{X_n}((\Tc_0,\Fc_0),(\Tc_1,\Fc_1))=\TE(\Phi)
%\end{equation}
%where $\Phi:\Tc_0\to\Tc_1$ is any affine diffeomorphism  in the homotopy class $\Fc_1\circ \Fc_0^{-1}$.
\item Example \ref{example:notmeasuerpres} below shows that the requirement in Part~(1) of  Theorem \ref{theorem:newmain} that maps be volume-preserving cannot be dropped.

\item We apply Theorem \ref{theorem:newmain} 
below to Kummer orbifolds (see Theorem \ref{theorem:kummermain}).  We hope to extend these results to the entire $57$-dimensional Teichm\"uller space of marked, Ricci-flat K3-orbifold metrics, 
which is a totally geodesic, symmetric submanifold of the  $X_{22}$. See \S\ref{section:HSP} for a discussion.
\end{enumerate}
\end{remarks}

\subsection{Dilatations  and Hilbert-Schmidt energies of a Lipschitz map} The rest of this introduction is devoted to defining the notions  (such as   $\Dil_k, \TE, \dots$) which appear above. We first need some elementary facts about inner product spaces. 

\medskip
\noindent
{\bf Linear algebra.} Let $E$ and $F$ be finite-dimensional inner product spaces and put $r:=\min\{\dim E,\dim F\}$.  Recall that for $A: E\to F$ is a linear map, its {\em operator norm} $\|A\|$  is defined as 
\[
\|A\|:=\sup_{v\in E\ssm \{0\}}\frac{\displaystyle |A(v)|_F}{\displaystyle |v|_E}.
\]

The  inner product on $E$ determines an inner product on $\wedge^k E$ for each $k\geq 1$ 
by declaring the basis $\{e_{i_1}\wedge \cdots \wedge e_{i_k} \}_{1\le i_1<\cdots <i_k\le \dim E}$ for $\wedge^kE$ to be orthonormal for $\{e_i\}$ an orthonormal basis for $E$. Likewise for $\wedge^k F$.  This gives an operator norm
$\|\; \|$ on $\Hom(\wedge^k E, \wedge^k F)$, which we make concrete in the case of interest as follows: for each  $A\in \Hom(E,F)$  there exists an orthonormal basis $(e_1, e_2, \dots )$ of $E$  such that $A(e_1),A(e_2), \dots$ are pairwise orthogonal (possibly zero) and have $F$-norm $\sigma_1\ge \sigma_2\ge  \cdots\ge 0$. This brings $A$ in `orthogonal diagonal form'. The sequence $\sigma_1\ge  \cdots\ge \sigma_r\ge 0$ only depends on $A$ and is called the \emph{singular sequence} of $A$.  So $\sigma_k=0$ for $k>r$ and $A$ has maximal rank $r$ if and only if $\sigma_r>0$. Then 
\[
\|\wedge^kA\|=\sigma_1\sigma_2\cdots\sigma_k.
\]

The inner product $\la A, B\ra:=\tr (B^*A)$ on $\Hom(E,F)$ 
yields another norm, the {\em Hilbert-Schmidt} norm 
\[\| A\|_{\HS}:=\sqrt{\tr (A^*A)}.\]  
Similarly this gives a norm $\|\wedge^kA\|_{\HS}$. Since  $A^*A(e_k)=\sigma_k^2 e_k$, we see that $\| A\|_{HS}=\sum_k\sigma^2_i$ and more generally, that 
\[
\textstyle \Pc_A(t):=\det (\Id_E+tA^*A)= \prod_k (1+t \sigma_k^2)=\sum_k t^k \|\wedge^kA\|^2_{\HS}.
\]

\medskip
\noindent
{\boldmath${\Dil_k}$} {\bf and} {\boldmath$\En_k$}.
We now bring the above into the context of  differential geometry. For each $k\geq 1$, the {\em $k$-dilatation} of a Lipschitz map $f:M\to N$ between Riemannian manifolds  with $M$ compact is 
\[
\Dil_k(f):=\esssup_{p\in M}\|\wedge^k Df_p\|.
 \]
Here $Df_x$ is defined almost everywhere by Rademacher's theorem. Thus
$\Dil_1(f)$ is the usual Lipschitz constant of $f$ and $\Dil_k(f)$ is the maximal
infinitesimal expansion of $k$-dimensional volume.  For convenience, we set 
$\Dil_0(f)=1$ for all $f$. It is clear  that $\Dil_k(f)=0$ when $k>r:=\min\{\dim M,\dim N\}$. We also note that 
when $n=\dim M=\dim N$, then $\Dil_n(f)=1$ when $f$ is volume-preserving.    

Such $k$-dilatations go back at least to Gromov  (\S 4.4 of \cite{Gr82}). We here use the terminology and
formulation of $k$-dilation developed systematically by Guth in \cite[\S 2]{Gu13}. We define an $L^2$-version, 
which we shall call the  \emph{$k$th Hilbert-Schmidt energy} of $f$ by  
\[
\textstyle \En_k(f):= \int_{M} \|\wedge^k Df_p\|_{\HS}^2\, d\!\vol_M.
\]
So $\En_0(f)=\vol(M)$ and $\En_k(f)=0$ for $k>r$. 

In Theorem \ref{thm:extremal} we show that for flat tori $M=\Tc_0$ and $N=\Tc_1$ (not necessarily of the same dimension),  each  $\Dil_k(f)$ and $\En_k(f)$  is minimized (but not always uniquely - see Example \ref{example:notmeasuerpres}) among all Lipschitz maps $f:\Tc_0\to \Tc_1$ in a fixed homotopy class by the affine maps in that homotopy class. 

\medskip
\noindent
{\bf Teichm\"uller potentials.}
Both $\Dil_k$ and $E_k$ have associated Teichm\"uller potentials, which we now define.

\begin{definition}[{\bf Total expansion}]\label{def:stretch}
Let $f:M\to N$ be a Lipschitz map between Riemannian manifolds with $M$ compact and put  $r:=\min\{\dim M,\dim N\}$. 

The {\em total expansion} of 
$f$ is
\[
\textstyle \TE(f):=
\begin{cases}
\sqrt{\sum_{k=1}^r\big|\log\Dil_k(f)-\log\Dil_{k-1}(f)\big|^2} &\text{if  $Df$ has somewhere rank $r$ and}\\
+\infty & \text{otherwise.}
\end{cases}
\]
(Note that in the first case, $\Dil_k(f)>0$ for $k=1, \dots, r$, so that then $\log \Dil_k(f)$ is defined.)
\end{definition}

Theorem \ref{theorem:newmain} shows that the total expansion  gives an extremal mapping interpretation of the metric $d_{X_n}$ on the symmetric space $X_n$.  As explained in 
Remark \ref{remark:2dcase}, the case $n=2$ of Theorem \ref{theorem:newmain} is equivalent 
to Teichm\"{u}ller's result  (at least if we confine ourselves to Lipschitz maps) that recovers  the hyperbolic metric $d_{X_2}$ in terms of 
minimal quasiconformal dilatation.  
\smallskip

It is clear from the definition that the Hilbert-Schmidt energies appear as coefficients of what we might call the 
\emph{Hilbert-Schmidt energy polynomial}:

\[
\textstyle \Pc_f(t):=\int_M \Pc_{Df} d\!\vol_M =\sum_k \int_M\|\wedge^k Df\|^2_{\HS}\, d\!\vol_Mt^k = 
\sum_k  \En_k(f) t^k\in \Rb^+[t].
\]

\medskip
The function $\Hc_n$ in the last clause of Theorem \ref{theorem:newmain} is defined as follows. 
For  a polynomial of the form 
$P(t)=1+e_1t+\cdots +e_{n}t^{n}$  with each $e_k\ge 0$ and $e_n>0$ we define 
\begin{equation}\label{eqn:Hn}
\textstyle \Hc_n(P)=\Hc_n (e_1, \dots, e_{n}):=\frac{1}{2}\sqrt{\sum_{k=1}^n (\log y_k)^2},
\end{equation}
where the $y_k$ are defined by the decomposition $P(t)=(1+y_1t)\cdots (1+y_nt)$ (so that $e_k$ becomes  the $k$th elementary symmetric function of $y_1, \dots, y_n$). Since the $e_k\ge 0$ and $e_n>0$, the $y_k$ are positive or come in complex conjugate pairs.
In either case, they lie in the domain $\Cb\ssm (-\infty, 0]$ of the principal branch of the logarithm. This  implies  that 
$\sum_{k=1}^n (\log y_k)^2$ is defined and also that it is real and positive. That $\Hc_n$ is strictly monotone in each variable $e_i$ (fixing the others) is a remarkable property due to Borisov-Neff-Thiel \cite[Lemma 4.1]{BNST17}  and which has a Mathoverflow origin. As the reader will see,  this property  seems almost tailor-made for what we will need. 

\begin{definition}[{\bf Hilbert-Schmidt expansion}]\label{def:Hstretch}
For $f:M\to N$  as in Definition \ref{def:stretch} its  {\em Hilbert-Schmidt  expansion} is 
\[
\textstyle \HE(f):=
\begin{cases}
\Hc_r(\En_1(f), \dots, \En_r(f)) &\text{if  $Df$ has somewhere rank $r$;}\\
+\infty & \text{otherwise.}
\end{cases}
\]
\end{definition}

\begin{remark}\label{rem:qc}
When the inner product spaces $E$ and $F$ have the same dimension $n$ and $A$ is a linear isomorphism, then  the rescaled map $\|\wedge^n A\|^{-1/n} A$ is volume  preserving.  When we apply the  above notions  to it, 
$\|\wedge^kA\|$ gets replaced by 
\[
\|\wedge^kA\|^{-k/n}.\|\wedge^kA\|=
(\sigma_1\sigma_2\cdots\sigma_k)^{1-k/n}(\sigma_{k+1}\cdots \sigma_n)^{-k/n}
\]
and likewise 
for  $\|\wedge^kA\|^2_{\HS}$. These expressions are homogeneous of degree zero and  depend only on the  conformal structures on $E$ and $F$ underlying their inner products.
We hit on a classical case when $n=2$ and $E$ and $F$ come only  endowed with a conformal structure: then the square of $\sigma_1^{1/2}\sigma_2^{-1/2}$ (so  $\sigma_1/\sigma_2$) is  usually denoted $K(A)$.

Suppose $M$ and $N$  are Riemann manifolds of the same dimension $n$ with $M$ compact and of unit volume.
If $f$ is a Lipschitz  map that has rank $n$ almost everywhere, then we can use the rescaled version of $Df$ to define 
corresponding expressions  $\Dil'_k(f)$ and $\En'_k(f)$. These are then conformal invariants and  hence so are  the  associated expansions $\TE'(f)$  and $\HE'(f)$. 

For example, for a  homotopy class  $\Fc$ of  maps $M\to N$ between closed connected hyperbolic surfaces (=having constant negative curvature) of unit volume   we get
\begin{equation}\label{eqn:comparison}
\textstyle \TE'(\Fc)=\inf_{f}\sqrt{2}\esssup_f \log K(Df)\ \ \text{and} \ \ \HE'(\Fc):=\frac{1}{2}\inf_{f}\int_M \log K(Df) d\!\vol_M, 
\end{equation}
where $f$ runs over the Lipschitz maps in $\Fc$. If $\Fc$ is a homotopy equivalence, then one can show that the restriction to Lipschitz homeomorphisms does not affect  the infima in \eqref{eqn:comparison},  and we then from $\TE'$ we get $\sqrt{2}$ times  the Teichm\"uller distance, and from $\HE'$ we can obtain the Riemannian form of the Weil-Petersson distance; see 
the end of Section \ref{sect:proofs} and Example \ref{example:hyperbolic}.
\end{remark}

\subsection{Computational resource disclosure and acknowledgements} The idea of using $k$-dilatations and Hilbert-Schmidt energies to arrive at Theorem \ref{thm:extremal}  came up  in a  conversation with  ChatGPT Pro. The same applies to Example \ref{example:notmeasuerpres}. Claude Pro was involved in checking the proofs.  It is a pleasure to thank Curt McMullen for asking a question that led us to define the Hilbert-Schmidt expansion $\HE(f)$, and for helpful comments and corrections on an earlier version of this paper. Farb is supported in part by NSF Grants DMS-181772 and DMS-2506566.

\section{Proof of Theorem \ref{theorem:newmain}}\label{sect:proofs}

In this section we prove Theorem \ref{theorem:newmain}.  Along the way we show that $\Dil_k(f)$ and $\En_k(f)$ are minimized on the affine maps in the homotopy class of $f$.

\subsection{Setup}  As mentioned above, the $n$-dimensional flat torus $\Tc$ is a principal homogeneous space for  the identity component of its  isometry group, referred to as its {\em group of  translations}. This group can be canonically identified with $H_1(\Tc; \Rb/\Zb)$ and so this  identifies each 
tangent space $T_p\Tc$ with $H_1(\Tc; \Rb)$. The inner product on $H_1(\Tc; \Rb)$ thus obtained  is independent of $p$.

This is functorial in the following sense:  given two flat  tori $\Tc_0$ and $\Tc_1$, a homomorphism 
\[\varphi: H_1(\Tc_0)\to H_1(\Tc_1)\] determines  a homotopy class $\Fc_\varphi$ of maps 
$\Tc_0\to\Tc_1$ inducing $\varphi$.  The class $\Fc_\varphi$ then contains an affine map $\Phi: \Tc_0
\to \Tc_1$, i.e.,   equivariant with respect to the homomorphism $H_1(\Tc_0; \Rb/\Zb)\to H_1(\Tc_1; \Rb/\Zb)$  between their translation groups induced by $\varphi$. Such a $\Phi$ is unique up to translation. Via the above identification, the derivative at any $p\in \Tc_0$ is simply 
$\varphi_\Rb: H_1(\Tc_0; \Rb)\to H_1(\Tc_1; \Rb)$. This is a map between inner product spaces and, as recalled in the introduction, such a map has singular sequence $\sigma_1\geq \sigma_2\geq\cdots\geq \sigma_n\geq 0$.
It is then clear that 
\[
\Dil_k(\Phi)=\|\wedge^k \varphi\|=\sigma_1\cdots\sigma_k.
\]
Let $r:=\min\{\dim \Tc_0, \dim\Tc_1\}$. So if $\Dil_r(\Phi)\not=0$,  then we are in the first case of Definition \ref{def:stretch} and
\[
\textstyle \TE(\Phi)=\sqrt{\sum_{k=1}^r (\log \sigma_k)^2}.
\]
In the special case where $\Tc_0$ and $\Tc_1$ have the same dimension $n$ and $\varphi$ is injective, $\Phi$ will be volume-preserving (and so $\Dil_n(\Phi)=1$)  if and only the  absolute degree of $\Phi$ (which is the order of the cokernel of $\varphi$) is equal to the volume ratio  $\vol(\Tc_0)/\vol(\Tc_1)$.

\subsection{Minimizing \boldmath$\Dil_k$ and $ \En_k$} Our first result is that the affine maps minimize $\Dil_k$ and $ \En_k$ in their homotopy class.  Both the statement and the proof for $\Dil_k$  fit in what might be called the  standard comass/calibration tradition. For example,  %Lawson \cite{Law75} and 
Liu \cite[Thm.\ 2.1]{Liu12} found that a dilatation bound for a map from a closed Riemannian manifold to a flat torus, can impose restrictions on its homotopy class.  The proof for $\En_k$ will be more elementary, as  it exhibits this invariant as an orthogonal projector in an inner product space. Our result is the following.

\begin{theorem}[{\bf Affine maps minimize $\Dil_k$ and $\En_k$ in their homotopy class}]
\label{thm:extremal}
Let $\Tc_0$ and $\Tc_1$ be flat  tori, not necessarily of the same dimension.  Let $\varphi: H_1(\Tc_0)\to H_1(\Tc_1)$ be a linear map and denote 
by $\Fc_\varphi$ the associated homotopy class of maps $\Tc_0\to \Tc_1$.
Then for any Lipschitz map 
$f\in \Fc_\varphi$, any affine $\Phi\in \Fc_\varphi$ and any $k\ge 1$, 
\[\Dil_k(f)\geq \Dil_k(\Phi)=\|\wedge^k\varphi\| \quad \text{and}\quad \En_k(f)\ge  \En_k(\Phi)=\|\wedge^k\varphi\|_{\HE}^2\vol(\Tc_0).
\]
Furthermore, if $\En_1(f)= \En_1(\Phi)$ then $f$ differs from $\Phi$ by a translation (hence is affine).
\end{theorem}

As mentioned in the introduction, the last assertion holds for every $\En_k, k\geq 1$. We omitted this from the statement, as this would need a separate proof. It is not needed for what follows.

\begin{remark}
The assertion about $\Dil_1$ of Theorem \ref{thm:extremal} for tori of the same dimension and of unit volume  is due to Greenfield--Ji \cite[Proposition~4.1]{GreenfieldJi21}.  Our proof restricts to $k=1$ to give a new proof of their result. 
\end{remark}

For the proof of Theorem \ref{thm:extremal} we need a linear algebra observation.  Let $E$ and $F$ be inner product spaces of dimension $m$ and $n$ respectively. Assume $E$ is oriented.  Then the star operator is defined and gives an 
isometric identification $*: \wedge^kE^\ast\xrightarrow{\cong}\wedge^{m-k}E^\ast$ and so 
$\|\wedge^k A\|$ is then also characterized as the smallest constant $c\ge 0$ for which 
\begin{equation}\label{eqn:}
|\alpha\wedge A^*\beta |_E\le c|\alpha|_E.|\beta|_F
\end{equation}
for all $\alpha\in \wedge^{m-k}E^*$ and $\beta\in \wedge^{k}F^*$.

\begin{proof}[Proof of Theorem \ref{thm:extremal} for $\Dil_k$] 
Let $r:=\min\{\dim\Tc_0,\dim\Tc_1\}$.  To prove the theorem, we can of course assume that $1\leq k\leq r$ and that $\Dil_k(\Phi)\neq 0$, as the other cases are trivial.

Let $m:=\dim \Tc_0$ and choose an orientation for $\Tc_0$ so that the  translation-invariant volume element $d\!\vol_{\Tc_0}$ can 
be regarded as an $m$-form on $\Tc_0$. Let $\alpha$ be a nonzero translation-invariant $(m-k)$-form on $\Tc_0$ and $\beta$ a nonzero translation-invariant $k$-form on $\Tc_1$.    

By the homotopy formula for flat currents under locally Lipschitz maps
\cite[\S\S 4.1.9 and 4.1.14]{Fed69}, the integral
\[
\textstyle   \int_{\Tc_0}\alpha\wedge f^*\beta
\]
depends only on the Lipschitz homotopy class of  $f$. Indeed, this follows
by applying the homotopy formula to the graph maps
$x\mapsto (x,f(x))$ and to the closed form
$\pi_0^*\alpha\wedge\pi_1^*\beta$ on  $\Tc_0\times \Tc_1$.  Since $f$ and $\Phi$ are both Lipschitz and homotopic, they are homotopic via some Lipschitz homotopy; see, e.g., Corollary 5.18 on p.109 
of \cite{LV77}.  It follows that
\begin{equation}\label{eqn:calibration}
\textstyle \int_{\Tc_0} \alpha \wedge f^*\beta=\int_{\Tc_0} \alpha \wedge \Phi^*\beta.
\end{equation}
Since   $\alpha$ and $\beta$ are translation invariant, their pointwise norms are constant  with value $|\alpha|_{\Tc_0}$ resp.\  $|\beta|_{\Tc_1}$.  It is then easy to find  $\alpha$ and $\beta$ such that the translation-invariant $m$-form $\alpha \wedge \Phi^*\beta$ equals $\Dil_k(\Phi) |\alpha|_{\Tc_0} .|\beta|_{\Tc_1}.d\!\vol_{\Tc_0}$. Integrating gives
\[
\textstyle \int_{\Tc_0} \alpha \wedge \Phi^*\beta =\Dil_k(\Phi) |\alpha|_{\Tc_0} .|\beta|_{\Tc_1}. vol(\Tc_0).
\]
On the other hand, it is clear that 
\[
\textstyle |\int_{\Tc_0} \alpha \wedge f^*\beta|\le  \Dil_k(f) |\alpha |_{\Tc_0} . |\beta |_{\Tc_1}.vol(\Tc_0)
\] 
and so $\Dil_k(\Phi)\le \Dil_k(f)$.
\end{proof}

\begin{proof}[Proof of Theorem \ref{thm:extremal} for $\En_k$] 
By identifying each tangent space of $\Tc_i$ with  $H_1(\Tc_i;\Rb)$, the derivative of $f$ becomes an integrable vector-valued  map  
\[Df: \Tc_0\to \Hom (H_1(\Tc_0;\Rb), H_1(\Tc_1;\Rb)).\]  

Likewise for 
\[\wedge^k Df: \Tc_0\to \Hom (\wedge^kH_1(\Tc_0;\Rb), \wedge^kH_1(\Tc_1;\Rb)).\] 

We can interpret the identity \eqref{eqn:calibration} as saying that the $\Hom (H_1(\Tc_0;\Rb), H_1(\Tc_1;\Rb))$-valued integrals
$\int_{\Tc_0} \wedge^k Df\,d\!\vol_{\Tc_0}$ and $\int_{\Tc_0} \wedge^k d\Phi\,d\!\vol_{\Tc_0}$
have  the same  matrix coefficients (such a coefficient being given by the pair $(\alpha, \beta)$) and so these integrals are equal. Since the second integrand is the constant $\wedge^k\varphi$, the following 
identity holds in $\Hom (H_1(\Tc_0;\Rb), H_1(\Tc_1;\Rb))$:
\begin{equation}\label{eqn:average}
\textstyle \int_{\Tc_0} \wedge^k Df\,d\!\vol_{\Tc_0}=\int_{\Tc_0} \wedge^k D\Phi\,d\!\vol_{\Tc_0}=\vol(\Tc_0).\wedge^k\varphi .
\end{equation}

Without loss of generality we assume that $\vol(\Tc_0)=1$. If $H$ is a finite-dimensional inner product space, then the space  of  Lipschitz maps  $u: \Tc_0\to H$ has the inner product  
\[\langle u , v\rangle_{\Tc_0}:=\int_{\Tc_0}\langle u , v\rangle_Hd\!\vol_{\Tc_0}\]  
for which the integration map $u\mapsto \int_{\Tc_0} u\, d\!\vol_{\Tc_0}$ defines  an orthogonal  projection $u\mapsto \bar u$ onto its subspace of constant maps $ \Tc_0\to H$. In particular, $\| u\|^2_{\Tc_0}=\|  \bar u\|^2_{\Tc_0}+\|  u-\bar u\|^2_{\Tc_0}$.

We apply this to $H=\Hom (\wedge^kH_1(\Tc_0; \Rb), \wedge^kH_1(\Tc_1; \Rb))$ with its Hilbert-Schmidt inner product
$\langle  A, B\rangle_{\HS}=\tr (B^*A)$ and  $u=\wedge^k Df: \Tc_0\to H$.   
It follows from the preceding that $\bar u=\wedge^k \varphi$ and hence 
\[
\En_k(f)=\|\wedge^k Df\|^2_{\Tc_0}= \|\wedge^k \varphi\|^2_{\HS}+\|\wedge^k Df-\wedge^k \varphi \|^2_{\Tc_0}=\En_k(\Phi)+\|\wedge^k Df-\wedge^k \varphi \|^2_{\Tc_0}\ge \En_k(\Phi).
\]
It also follows that $\En_k(f)=\En_k(\Phi)$ implies $\wedge^k Df=\wedge^k \varphi $; for $k=1$, this means that $Df$  is constant equal to $\varphi$, which implies  that $f$ and $\Phi$ differ by a translation.
\end{proof}

\subsection{A linear algebra lemma}
We will need a linear algebra lemma. This in its turn, needs the following standard  fact (see for instance 
\cite[Ch.\ 1]{MOA11}); for completeness we give a proof.

\begin{lemma}\label{lemma:ineq1}
Let $a_1\ge \cdots \ge a_n>0$ and $b_1\ge \cdots \ge b_n\ge 0$ be monotone nonincreasing sequences of positive resp.\ nonnegative real numbers with the property that
\[
b_1+\cdots +b_k\le a_1+\cdots +a_k, \quad  \text{for all $k=1, \dots, n$}. 
\]
Then  $\sum_{k=1}^n a_kb_k\le \sum_{k=1}^n a_k^2$  and this inequality is strict unless $a_k=b_k$ for all $k=1, \dots, n$.
\end{lemma}

\begin{proof}
For \(0\leq k\leq n\), set $A_k:=\sum_{i=1}^k a_i$ and $B_k:=\sum_{i=1}^k b_i$. So $B_k\le A_k$ by hypothesis  and 
$A_0=B_0=0$.  Then 
\begin{multline}
\textstyle  \sum_{k=1}^n a_kb_k= \sum_{k=1}^n a_k(B_k-B_{k-1}) =a_nB_n+ \sum_{k=1}^{n-1}(a_k-a_{k+1})B_k\le \\
\textstyle  \le a_nA_n+ \sum_{k=1}^{n-1}(a_k-a_{k+1})A_k=\sum_{k=1}^n  a_k(A_k-A_{k-1})=\sum_{k=1}^n a_k^2
\end{multline}
Equality forces $B_n=A_n$ and $B_k=A_k$ when $a_k>a_{k+1}$. Tracking this through (using that the $(b_k)_k$  is also nonincreasing)
then yields  $b_k=a_k$ for all $k$.
\end{proof}

\begin{lemma}\label{lemma:linalg}
Let $E$ and $F$ be real inner product  spaces of the same dimension $n$ and let $A: E\to F$ and $B: E\to F$ be linear maps with $A$ an isomorphism.  Suppose 
that $\|\wedge^k B\|\le  \|\wedge^kA\|$ for $k=1, \dots, n$. Then 
\[\tr(A^*B) \le \tr(A^*A)\] 
and if equality holds then $B=A$.
\end{lemma}
\begin{proof}
Let $(a_1\ge\cdots\ge a_n>0)$ be the singular sequence of $A$ and $(b_1\ge\cdots\ge b_n\ge 0)$ the one of 
$B$. Von Neumann's  trace inequality asserts that  then $\tr(A^*B)\le \sum_{k=1}^n a_k b_k$ (see for example \cite{Mirsky75}  for a short proof). 

The assumption of the lemma states that $b_1\cdots b_k\le a_1\cdots a_k$ for $k=1, \dots, n$.  It follows from this  that 
\begin{equation}\label{eqn:ineq2}
b_1+\cdots +b_k\le a_1+\cdots +a_k, \quad  \text{for all $k=1, \dots, n$}
\end{equation}
(this  is known among statisticians as ``log-majorization implies weak-majorization'', see for example \cite[Ch.\ 5 A.2.b]{MOA11}). 
The inequalities  \eqref{eqn:ineq2} and  Lemma \ref{lemma:ineq1} yield 
\[\label{eqn:ineq3}
\textstyle \tr (A^*B)\le \sum_{k=1}^n a_k b_k\le \sum_{k=1}^n a_k^2 =\tr(A^*A).
\]
and that equality implies  $b_k=a_k$ for $k=1, \dots, n$. So it remains to show that in that last case we must have $B=A$.
This follows from the Cauchy-Schwarz inequality for the  Hilbert-Schmidt inner product on $\Hom(E,F)$, for 
$\tr (A^*B)$, $\tr(A^*A)$ and $\tr(B^*B)$  then all take  the same value $\sum_{k=1}^n a_k^2$.
\end{proof}

\subsection{Proof of Theorem \ref{theorem:newmain}}
\begin{proof}[Proof of Part~(1) of Theorem \ref{theorem:newmain}]
Let $f:  \Tc_0\to \Tc_1$  be  a volume-preserving Lipschitz map of absolute degree $d=\vol(\Tc_0)/\vol(\Tc_1)$ and let $\varphi: H_1(\Tc_0)\to H_1(\Tc_1)$ the map induced on homology.  Let 
$a_1\ge a_2\ge\cdots \ge a_n>0$ be the singular sequence of 
$\varphi: H_1(\Tc_0; \Rb)\to H_1(\Tc_1; \Rb)$, and let $\Phi:\Tc_0\to\Tc_1$ be an affine map homotopic to $f$. 

We put $\lambda_k:=\log a_k$. So $\lambda_k\ge \lambda_{k+1}$ and 
$\log \Dil_k(\Phi)=\lambda_1+\cdots +\lambda_k$. Since  $f$ and $\Phi$ are volume-preserving, 
\[\Dil_n(f)=\Dil_n(\Phi)=a_1a_2\cdots a_n=1\] 
and $\Dil_0(f)=\Dil_0(\Phi)=1$ by definition. 
By Theorem~\ref{thm:extremal}, $\Dil_k(f)\ge \Dil_k(\Phi)$ and so $t_k:=\log  \Dil_k(f)-\log \Dil_k(\Phi) \ge 0$.  Then
\begin{multline*}
\textstyle     \TE(f)^2-\TE(\Phi)^2 =\\
\textstyle = \sum_{k=1}^n (\log D_k(f)-\log D_{k-1}(f) )^2 -  \sum_{k=1}^n (\log D_k(\Phi)-\log D_k(\Phi))^2=\\
\textstyle =\sum_{k=1}^n  \bigl( (\lambda_k+t_k-t_{k-1})^2 -\lambda_k^2  \bigr) =\sum_{k=1}^n (t_k-t_{k-1})^2  +
2\sum_{k=1}^n\lambda_k(t_k-t_{k-1}) 
\end{multline*}
and since Abel summation gives  
\[
\textstyle  2\sum_{k=1}^n\lambda_k(t_k-t_{k-1})=2\sum_{k=1}^{n-1}(\lambda_k -\lambda_{k+1})t_k\ge 0,
\]
 this proves that $\TE(f)\geq \TE(\Phi)$.

Suppose now that $\TE(f)=\TE(\Phi)$. Then the preceding shows that $\sum_{k=1}^n(t_k-t_{k-1})^2=0$ and hence $t_k=t_{k-1}$ for every $k$. Since 
$t_0=0$, it follows that $t_k=0$ for all $k$, so that $\Dil_k(f)=\Dil_k(\Phi)$ for $k=1, \dots ,n$.
We regard the derivatives $Df_p$ (where defined) and $d\Phi_p$ as maps between the tangent spaces of $\Tc_0$ and $\Tc_1$ at their identity element (so that now $d\Phi_p=\varphi$).  For almost every $p\in \Tc_0$ where $f$ is smooth, Lemma \ref{lemma:linalg} with $A=\varphi$ and $B=Df_p$ yields
\begin{equation}\label{eqn:trace-pointwise}
\tr (\varphi^*Df_p)\le \tr (\varphi^*\varphi)
\end{equation}
or  equivalently,  $\tr (\varphi^*(\varphi-Df_p))\ge 0$. The identity   \eqref{eqn:average}  gives that  $\int_{\Tc_0} (Df-\varphi)  d\!\vol_{\Tc_0}=0$
and so the integral of the nonnegative function $\tr (\varphi^*(\varphi-Df_p))$ is zero. This implies that $\tr (\varphi^*(\varphi-Df_p))$ is zero almost everywhere. Invoking Lemma \ref{lemma:linalg} once more, we then find that $Df_p$ is constant equal to 
$\varphi$ almost everywhere. This implies that $f$ differs from $\Phi$ by a translation, hence is affine. 

The remaining assertion,  that  $\HE(f)=\HE(\Phi)$ yields the same conclusion, will be proved later.
\end{proof}

We next describe the normalized Riemannian metric on $X_n$ and its associated path metric. A point of $X_n$ is a  determinant-one positive-definite 
quadratic form on $\Rb^n$.   We use the normalization of the standard invariant metric for which a tangent vector at $g$ is a symmetric bilinear form $u$ satisfying $\operatorname{tr}(g^{-1}u)=0$ and
\begin{equation}
\label{eq:standard-metric-Xn}
\textstyle \langle u,v\rangle_g :=\frac14\operatorname{tr}\big((\tilde g^{-1}u)(\tilde g^{-1}v)\big),
\end{equation}
where $\tilde g\in \GL_n(\Rb)$ is the self-adjoint invertible transformation associated to $g$.
Equivalently, in the quotient model $X_n=\SL(n,\Rb)/\SO(n)$, the tangent space at the base point is the space of symmetric trace-zero matrices and the inner product there is $(S,T)\mapsto\frac14\operatorname{tr}(ST)$.  For $g_0,g_1\in X_n$, choose a $g_0$-orthonormal basis $(e_1,\ldots,e_n)$ 
of $\Rb^n$ which is $g_1$-orthogonal.  So $g_1(e_i,e_j)=e^{2\lambda_i}\delta_{ij}$ for some $\lambda_i\in \Rb$ (equivalently,  $(e^{-\lambda_i}e_i)_{i=1}^n$ is orthonormal for $g_1$).
Since $g_0$ and $g_1$ have determinant one, $\sum_{i=1}^n\lambda_i=0$.  The curve $g_t$, $0\leq t\leq 1$, defined by
\[
g_t(e_i,e_j)=e^{2t\lambda_i}\delta_{ij}
\]
is the symmetric-space geodesic from $g_0$ to $g_1$.
By \eqref{eq:standard-metric-Xn}, its squared speed is $\sum_{i=1}^n\lambda_i^2$. 
Since $X_n$ is a Hadamard manifold, this geodesic is minimizing and hence
\begin{equation}
\label{eq:distance-Xn-singular-values}
\textstyle d_{X_n}(g_0,g_1)=\sqrt{\sum_{i=1}^n\lambda_i^2}.
\end{equation}

We now turn to the proof of Part~(2) of Theorem \ref{theorem:newmain}. As explained in the first item of  Remarks \ref{rems:main}, this comes down to
proving that $\TE$ and $\HE$ take on  the identity map on the torus $\Rb^n/\Zb^n$ endowed with the flat metrics defined
by the determinant 1 inner products $g_0$ resp.\ $g_1$ on $\Rb^n$ (denoted $\Phi$) the value $d(g_0,g_1)$.

\begin{proof}[Proof of Part~(2) of Theorem  \ref{theorem:newmain} for $\TE$]
Let $\varphi$ be the identity map of $\Rb^n$, but regarded as a map   $(\Rb^n,g_0)\longrightarrow (\Rb^n,g_1)$ between inner product spaces. Let  $e^{\lambda_1}\ge \cdots \ge e^{\lambda_n}$ be the singular sequence of $\varphi$.  So then
\[
\Dil_k(\Phi)=e^{\lambda_1+\cdots+\lambda_k}
\qquad (1\leq k\leq n).
\]
Note that $\Dil_n(\Phi)=1$ ($\Phi$ is volume-preserving).  It follows from Definition \ref{def:stretch} that
\[
\textstyle \TE(\Phi)^2
=\sum_{k=1}^n
\left|\log\Dil_k(\Phi)-\log\Dil_{k-1}(\Phi)\right|^2 =\sum_{k=1}^n\lambda_k^2.
\]
Together with \eqref{eq:distance-Xn-singular-values}, this gives
$\TE(\Phi)=d_{X_n}((\Tc_0, f_0),(\Tc_1,f_1))$.
By Part~(1), every volume-preserving Lipschitz map $h$ in $\Fc_1\circ \Fc_0^{-1}$ satisfies $\TE(h)\geq\TE(\Phi)$ and so
\[
\inf_{h\in  \Fc_1\circ \Fc_0^{-1}}\TE(h)
=\TE(\Phi)
=d_{X_n}((\Tc_0,\Fc_0),(\Tc_1,\Fc_1)),
\]
as claimed.
\end{proof}

\begin{proof}[Proof of Theorem \ref{theorem:newmain} for $\HE$]
Let $\varphi$ and $e^{\lambda_1}\ge \cdots\ge e^{\lambda_n}$ be as above.  
Let $f: \Tc_0\to \Tc_1$ resp.\ $\Phi: \Tc_0\to \Tc_1$ be a volume-preserving Lipschitz map  resp.\ an affine map, both  inducing $\varphi$. There are three things for us  to 
show: 

\begin{enumerate}
\item $\Hc_n(E_1(f), \dots , E_n(f))\ge \Hc_n(E_1(\Phi), \dots , E_n(\Phi))$.
\item Equality in (i) implies $f$ affine (so $f$ is then equal to $\Phi$ up to translation).
\item $\Hc_n(E_1(\Phi), \dots , E_n(\Phi))=\sqrt (\sum_{k=1}^n 
\lambda_k^2)$.
\end{enumerate}

Let us begin with noting that $\En_n(\Phi)=1$ (because $\varphi$ is an isomorphism) and that  $\En_n(f)=1$ (because $f$ is volume preserving).  By Theorem \ref{thm:extremal}, $\En_k(f)\ge \En_k(\Phi)$ for $k=1,\dots, n-1$ with 
$\En_1(f)= \En_1(\Phi)$ implying  that $f$ equals $\Phi$ up to a translation.

Recall that 
\[
\textstyle \sum_k \En_k(\Phi)t^k=\det (\Id_E+t\varphi^*\varphi)= \prod_k (1+t e^{2\lambda_k}).
\]
If  $y_1, \dots, y_n$ are defined by the property  that 
\[
\textstyle \sum_k \En_k(f)t^k=(1+ty_1)\cdots (1+ty_n),
\]
then by definition (Equation \eqref{eqn:Hn}), 
\[
\textstyle \Hc_n(E_1(f), \dots, E_n(f))=\frac{1}{2}\sqrt{\sum_k (\log y_k)^2}.
\] 
This ensures that we get for $f=\Phi$
the value  $\sqrt (\sum_k \lambda^2_k)$, which thereby settles (3). 

This brings us in a situation  where all the hypotheses are fulfilled of Lemma 4.1  of 
Borisov-Neff-Thiel  \cite{BNST17}. That lemma  then tells us that $\sum_k \lambda_k^2\le \sum_k (\frac{1}{2}\log y_k)^2$, proving (1) and that equality implies $\En_k(\Phi)=\En_k(f)$ for all $k$. By taking  $k=1$, we then conclude that $f$ differs from $\Phi$ by a translation, proving (2).
\end{proof}

\begin{example}[{\bf Dropping the volume preserving assumption destroys affine minimality}]
\label{example:notmeasuerpres}The following example was suggested to us by ChatGPT.
For $\lambda\ge 0$ we turn the  standard $2$-torus $\Rb^2/\Zb^2$ into a flat one  using the metric defined by
quadratic form $(e^\lambda x)^2 + (e^{-\lambda}y)^2$ on $\Rb^2$. The map $\Phi: \Tc_0\to \Tc_\lambda$ which induces the identity 
in $\Rb^2/\Zb^2$ has  the singular sequence $(e^\lambda , e^{-\lambda})$ and so 
\[
\Dil_1(\Phi) = e^{\lambda}\quad and \quad   \Dil_2(\Phi) = 1.
\]
This implies that $\TE(\Phi) = \sqrt{2}\lambda = d_{X_2}(\Tc_0,\Tc_\lambda)$.
Now define $f: \Tc_0\to \Tc_\lambda$ by 
\[
f(x,y) = (x, y + (e^\lambda - 1)(2\pi)^{-1}\sin(2\pi y)) \pmod{\Zb^2}.
\]
Then $Df$ has the diagonal form $(\begin{smallmatrix} 1 & 0\\ 0 & c(y)\\ \end{smallmatrix})$, where $c(y) = 1 + (e^\lambda - 1)\cos(2\pi y)$.
Let us take  $e^\lambda = 3/2$. Then  $0 < c(y) \le \lambda$ and so $f$ is a
smooth degree-one diffeomorphism in the same homotopy class as $\Phi$. Its local
singular values are $e^\lambda$ and $e^{-\lambda}c(y)$. Since $c(y) \le  e^\lambda < e^{2\lambda}$,
we have $\Dil_1(f) = e^\lambda$ and $\Dil_2(f) = \sup c(y) =e^\lambda=3/2$ and hence 
\[
\TE(f) = \sqrt{(\lambda^2 + (\lambda - \lambda)^2} =\lambda < \sqrt{2}\lambda = \TE(\Phi)
\]
and so the affine map $\Phi$ does not minimize $\TE$ in its homotopy class. 
\end{example}

\begin{remark}[{\bf Functionals that don't work}] \label{remark:dontwork}
Let $\Fc=\Fc_1\circ\Fc_0^{-1}$, let $\Phi\in\Fc$ be affine, and write its
singular sequence as $e^{\lambda_1}\geq\cdots\geq e^{\lambda_n}$, where
$\sum_i\lambda_i=0$.  Then
\[
\textstyle d_{X_n}((\Tc_0,\Fc_0),(\Tc_1,\Fc_1))=\sqrt{\sum_i\lambda_i^2}.
\]
Three other candidate extremal problems that could be used to reconstruct $d_{X_n}$ are the best Lipschitz constant $\Dil_1$, the quasiconformal dilatation $K$, and the total energy $\En_1(f):=\int_{\Tc_0}\|Df\|_{\rm HS}^2\,d\!\vol_{\Tc_0}$.  The infima of these functionals on a fixed homotopy class are realized on the affine maps, with values: 
\[
\textstyle \inf_f \Dil_1(f)=e^{\lambda_1},\qquad
\inf_f K(f)=e^{\lambda_1-\lambda_n},\qquad
\inf_f \En_1(f)=\sum_i e^{2\lambda_i}
\]
where the infima should be taken over the $f\in \Fcal$ that are  Lipschitz, quasiconformal homeomorphisms and Lipschitz respectively.  However,  as the following examples show, no  function of any of these functionals can give $d_{X_3}$. 
By appending zeroes, we obtain such examples for $d_{X_n}$ for any  $n\ge 3$. 
\begin{enumerate}
\item The logarithmic
singular sequences $(t,0,-t)$ and $(t,-t/2,-t/2)$ have the same minimal Lipschitz constant $e^t$, but symmetric-space distances $\sqrt2\,t$ and $\sqrt{3/2}\,t$.  
\item The sequences $(t,0,-t)$ and $(4t/3,-2t/3,-2t/3)$
have the same minimal quasiconformal dilatation $e^{2t}$, but $X_n$-distances
$\sqrt2\,t$ and $\sqrt{8/3}\,t$.  

\item The minimum energy alone does not determine the symmetric-space
distance. This already occurs in dimension \(3\). Let $z$ be the largest solution of the equation $x+x^{-1}=\frac{13}{4}$,
that is $z:=\frac{1}{8}(13+\sqrt{105}$, and let $t:=\frac{1}{2}\log z$. Consider the 
logarithmic singular sequences 
\[
\textstyle  A=(t,0,-t), \quad \text{and } B=(s,s,-2s)\quad \text{with } s=\frac12\log 2.
\]  
Both sequences have sum zero. 
For a singular sequence $(e^{\lambda_1}\ge e^{\lambda_2}\ge e^{\lambda_3})$ the minimum energy is
$\En_{\min}(\lambda)=\sum_{i=1}^3 e^{2\lambda_i}$. Consequently,
\[
\textstyle   \En_{\min}(A) =e^{2t}+1+e^{-2t} =z+1+\frac1z =\frac{17}{4},\quad  \En_{\min}(B =2e^{2s}+e^{-4s} =4+\frac14 =\frac{17}{4}.
\]
Thus the two homotopy classes have the same minimum energy. Their
squared symmetric-space distances, however, are
\[
\textstyle    d_A^2 =\lVert A\rVert_2^2 =2t^2 =\frac12(\log z)^2,\quad d_B^2 =\lVert B\rVert_2^2 =6s^2 =\frac32(\log 2)^2
\]
and these are clearly different. Therefore two homotopy classes can have the same minimum energy but
different symmetric-space distances. 
\end{enumerate}
\end{remark}

\begin{remark}[{\bf The $2$-dimensional case and the classical Teichm\"uller distance}]
\label{remark:2dcase}
Suppose $M$ and $N$ are Riemann surfaces with $M$ closed. 
Assume $f:M\to N$ is a Lipschitz map which is of rank $2$ almost everywhere.
If  $f$ is differentiable of rank 2 at $p\in M$, then  choose hermitian inner products  in $T_pM$ and $T_{f(p)}N$ so that  $D_pf$ has a singular sequence, say  $\sigma_1>\sigma_2>0$. The quotient $K(D_pf):=\sigma_1/\sigma_2$ is independent of these  choices and is what is usually called the {\em dilatation} of $f$ at $p$. The dilatation of $f$ is the  essential supremum of this function and denoted $K(f)$. Remark \ref{rem:qc}  shows that this is also the square of what we denoted $\Dil'_1(f)$.
Teichm\"uller assigned to the homotopy class $\Fc$ defined by $f$ the `distance'  $\frac{1}{2}\log\inf_f K(f)=\log \inf_f\Dil'_1(f)$ (where $f$ runs over this  homotopy class). Note that this also 
$\frac{1}{\sqrt{2}}\TE'(\Fc)$.

This  is discussed from the Lipschitz/Thurston-metric point of view for the case when $M$ and $N$ are closed Riemann surfaces of genus one by Belkhirat--Papadopoulos--Troyanov \cite{BPT05}, and for higher-dimensional
flat tori by Greenfield--Ji \cite{GreenfieldJi21}, where they give a modular interpretation of 
a ``projective Hilbert metric'' on $X_n$. 
\end{remark}

\section{The Hilbert-Schmidt expansion as a Riemannian metric}
\label{section:HSP}

For a manifold  $M$,  the space of Riemannian metrics on $M$ make up an open convex cone $\Riem(M)$ in the space of all  quadratic forms on its tangent bundle. So for a Riemannian metric $g$ on $M$, the tangent space $T_g\Riem(M)$ can be identified 
with the linear space of all quadratic forms on $TM$, or what amounts to the same thing, the space of $g$-selfadjoint endomorphisms of $TM$. The latter  space comes with a (standard) inner product defined by 
\[
\textstyle \la S, T\ra_g =\int_M \tr(ST)d\!\vol_g.
\]
Assume  $M$ is closed. For a real number $\mu>0$ denote by  $\Riem(M)_\mu\subset \Riem(M)$ the  Riemannian metrics on $M$ with volume $\mu$.
There is an evident isomorphism $\Rb_{>0}\times \Riem(M)_\mu\cong\Riem(M)$.
If $g\in \Riem(M)_\mu$, then the  tangent space of  $\Riem(M)_\mu$ at $g$ is  the hyperplane in the space of $g$-selfadjoint endomorphisms $S$ of $TM$ defined by 
\[
\textstyle \la S, \Id\ra_g=\int_T \tr(S)d\!\vol_g=0.
\]

\begin{proposition}[The infinitesimal Hilbert-Schmidt expansion]\label{prop:WP}
The infinitesimal version of the Hilbert-Schmidt expansion on the closed manifold $M$ is given by half  the quadratic form associated with the standard inner product: Let  $g\in \Riem(M)_\mu$ (with $\mu>0$) and let $\tau\in(-\varepsilon, \varepsilon)\mapsto h(\tau)$ be a smooth curve in $\Riem(M)_1$ with 
$h(0)=g$.  If  $f_\tau: (M,g)\to (M, h(\tau))$ is  given by the identity map, then 
\[
\textstyle \HE(f_\tau)^2\equiv \frac{1}{2}\tau^2 \int_M\tr(S^2)d\!\vol_M\equiv \frac{1}{2}\tau^2\la S, S\ra_g \pmod{\tau^3},
\]
where $S$ the self-adjoint endomorphism of the tangent bundle of $M$ defined by $\dot h(0)$.
So $\HE(f_\tau)$ is differentiable in $\tau$ and   $\frac{d}{d\tau}\big|_{\tau=0}\HE(f_\tau)=\frac{1}{\sqrt 2}\la S, S\ra_g$.
\end{proposition}
\begin{proof}
Let $H(\tau)$ be the selfadjoint endomorphism of the tangent bundle characterized by $h(\tau)(v,v')=g(H(\tau)(v), v')$.
So $H(\tau)\equiv\Id+\tau S\pmod{\tau^3}$.
Then the adjoint  $Df_\tau^*$ of $Df_\tau$ is equal to $H(\tau)$. 
If  $\{s_k(p)\}_{k=1}^n$ are the eigenvalues of $S_p$, then the  eigenvalues of $H(\tau)$ are $y_k(p)\equiv 1+\tau s_k(p)\pmod{\tau^2}$, $k=1, \dots, n$ and hence 
\[
\textstyle \sum_k (\log y_k(p))^2\equiv \tau^2 \sum_k s_k(p)^2\equiv\tau^2 \tr(S_p^2)\pmod{\tau^3}.
\] 
Since $\tr(S^2)$ is globally defined, the definition of $\HE$  shows  that 
\[
\textstyle \HE(f_\tau)^2\equiv\frac{1}{2}\tau^2 \int_M \tr(S^2)d\!\vol_g\equiv \frac{1}{2}\tau^2\la S, S\ra_g\pmod{\tau^3},
\] 
as asserted.
\end{proof}

One usually considers not $\Riem (M)_\mu$, but rather its orbit space with respect to the action of 
$\Diff(M)$ or of its identity component $\Diff(M)^\circ$.   A theorem of Ebin \cite{Ebin68} shows that the slice theorem 
for proper Lie group  actions on a finite-dimensional smooth manifold  generalizes to this setting so that in case 
$\Diff(M)^\circ$ acts freely, such slices can function as charts of the orbit space, 
giving it the structure of a \emph{ILH-manifold} (i.e., modeled on an inverse limit of Hilbert spaces) endowed with a  Riemannian metric in that category. The tangent space at the image of $g$ in this orbit space is the normal space to the $\Diff(M)^\circ$-orbit of $g$ and  consists of the  $S$ that are obtained as  the $g$-selfadjoint part of a Lie derivative  of a vector field on $X$.  This is still an infinite-dimensional space when $\dim M\ge 2$. The $\HE$-potential descends to this orbit space and Proposition \ref{prop:WP}  tells us that it induces its Riemannian metric.
% This however need not be the associated path metric.

However, by  restriction to a 
suitable subspace of Riemannian metrics, we may end up in some cases with a finite-dimensional manifold, which we then 
may regard as a Teichm\"uller space for $M$. Here are two basic examples. The first shows that $\HE$ recovers the Weil-Petersson Riemannian metric on the classical Teichm\"uller spaces (for  closed genus $\ge 2$ surfaces). 

\begin{example}\label{example:hyperbolic}
Let $M$ be a closed surface of  genus $g\ge 2$. We here  restrict to hyperbolic  Riemannian metrics of unit volume (so by Gau\ss-Bonnet these have  constant negative Gau\ss\ curvature $(2-2g)^{-1}$) and  so lie in $\Riem (M)_{1}$). The  $\Diff(M)^\circ$-orbit space  of this locus yields its Teichm\"uller space $\Teich(M)$ 
endowed with the Weil-Petersson metric. This is due to Fischer-Tromba \cite{FT84} (see also \cite[\S 3.2]{Yam14}). Note that we are not claiming that one can recover the induced 
Weil-Petersson path metric on $\Teich(M)$, but only the path metric on $\Teich(M)$ induced from that on the orbit space of $\Riem (M)_{1}$. The point is that $\Teich(M)$ is not isometrically embedded in this orbit space.
\end{example}

Other Riemannian manifolds are known to have a (locally) symmetric space as a moduli space. Among them are the Ricci-flat
Riemannian manifolds which come from a hyperk\"ahler structure: the symmetric space is then that of  a Lie group isomorphic with $\SO (3, r)$ for some $r$. If $G$ is a connected real semisimple group, then we can always embed it some $\SL(n,\Rb)$ such that
its intesection with $\SO(n)$ is a maximal compact subgroup. This embeds the symmetric space $X_G$  of $G$ in $X_n$ and thus one could excpect that if an an open subset of  $X_G$ parametrizes certain   Riemannian manifolds, then these manifolds will be directly related to the  flat tori we get via their embedding. The hope is then  that the metric on $X_G$ also arises as the solution of an extremal problem. We illustrate this with the Kummer construction.

\begin{definition}
Let  $\Tc$ be a flat torus. A \emph{Kummer involution} of $\Tc$ is an  isometric involution $\iota\in \aut(\Tc)$ which induces minus the identity in $H_1(\Tc)$. We call  the Riemannian orbifold defined by such an involution  (and any Riemannian orbifold  isomorphic to it)  a \emph{Kummer orbifold}.
\end{definition}

Observe that if  $n=\dim \Tc$, then a Kummer involution $\iota$ of $\Tc$ has $2^n$ fixed points and any two distinct fixed points differ by a translation  of order $2$: the fixed point set of $\iota$ is a $H_1(\Tc; \Fb_2)$-torsor and maps bijectively to the set of orbifold points of the associated  Kummer orbifold. Two  Kummer involutions of $\Tc$  are conjugate under a  translation.
Note that $\TE(\iota)=\HE(\iota)=0$. Indeed, the  terms of the singular sequence for minus the identity of an inner product space $E$ are all equal $1$ and hence so are all its dilatations and energies.

If $\Kc_0$ and $\Kc_1$ are Kummer orbifolds defined by $(\Tc_0, \iota_0)$  and  $(\Tc_1, \iota_1)$, then a  \empty{continuous map} $f:\Kc_0\to \Kc_1$ is by definition a map $\Kc_0\to \Kc_1$ which lifts to an equivariant  continuous map $\Tc_0\to \Tc_1$. So such a map must take orbifold points to orbifold points. This extends in an evident manner notions of a smooth map, a Lipschitz map, an affine map  and a homotopy class of maps between $\Kc_0$ and $\Kc_1$. The affine maps in a homotopy class differ by an order $2$ translation in $\Kc_1$ and so there are only finitely many such.

We can also interpret $X_n$ as the moduli space of marked  Kummer orbifolds of dimension $n$ of volume  $1$.

A  Lipschitz map $f:\Kc_0\to\Kc_1$ lifts by definition to an equivariant 
 Lipschitz map  $\tilde f: (\Tc_0, \iota_0)\to (\Tc_1, \iota_1)$ and this lift is unique up to postcomposition with $\iota_1$. These two  lifts  define the same singular sequence at every point $p\in \Tc_0$ where $f$ is smooth and so  have there the same dilatations and energies at $p$. It is clear that  $\Dil_k(f)=\Dil_k(\tilde f)$ and $\Ec_k(f)=\Ec_k(\tilde f)$ and hence $\TE(f)=\TE(\tilde f)$  and $\HE(f)=\HE(\tilde f)$.

Theorem \ref{theorem:newmain} has an obvious extension to the Kummer setting:

\begin{theorem}[{\bf Kummer version of the Main Theorem}] \label{theorem:kummermain}
Let $\Kc_0$ and $\Kc_1$ be flat Kummer orbifolds, both of dimension $n\ge 1$ and  of unit volume and let $\Fc$ be a homotopy class of homotopy equivalences $\Kc_0\xrightarrow{\sim} \Kc_1$.  Then  $\TE$ and $\HE$ have a minimum on the set of \emph{volume-preserving} Lipschitz  maps in $\Fc$ and takes that  value only  on the $2^n$ many affine maps  in $\Fc$.
The symmetric space distance is also the  Teichm\"uller distance for flat Kummer orbifolds of dimension $n$ defined by $\TE$ and $\HE$. 
\end{theorem}

We put this into context for  the classical Kummer dimension $4$. What makes this dimension  special is that a Kummer orbifold of dimension 4 can be regarded as a mildly degenerate K3-manifold (a K3 orbifold) endowed with a Ricci-flat metric: it fits in a $57$-dimensional family of 4-dimensional, closed oriented orbifolds endowed with a Ricci-flat metric of a fixed volume, whose general member is a K3 manifold  without orbifold points. The second homology group of a K3 manifold (endowed with its intersection pairing)  is an even unimodular lattice  $\Lambda$   of signature $(3,19)$  and the base of this family is 
$\Gras_3^+(\Lambda_\Rb)$.  Our  main theorem is compatible with an embedding of summetric spaces $X_4\hookrightarrow \Gras_3^+(\Lambda_\Rb)$.
To see this, consider  $\Rb^4$ as an oriented real vector  space, with the orientation given as an isomorphism $\det: \wedge^4\Rb^4\to \Rb$. This defines on 
$\wedge^2\Rb^4$  a quadratic form 
\[
\delta: \alpha\in\wedge^2\Rb^4\mapsto \det(\alpha\wedge\alpha)\in \Rb
\]  
which  is nondegenerate of signature $(3,3)$. 
Any positive definite quadratic form $g$   on $\Rb^4$ defines a star operator $*_g$ on $\wedge^2\Rb^4$.
This is an involution whose eigenspaces $\wedge_{g,\pm}^2\Rb^4\subset \wedge^2\Rb^4$ are perpendicular for $\delta$ with 
$\wedge_{g,+}^2\Rb^4\subset \wedge^2\Rb^4$ positive definite  and $\wedge_{g,-}^2\Rb^4\subset \wedge^2\Rb^4$ negative definite of dimension $3$.
The assignment $g\mapsto \wedge_{g,+}^2\Rb^4$ maps the space $X_4$ of inner products on $\Rb^4$ of determinant $1$ isomorphically onto 
the open subset  $\Gras_3^+(\wedge^2\Rb^4)\subset \Gras_3(\wedge^2\Rb^4)$ of positive definite $3$-planes in $V$. This identifies the symmetric space of $\SL(\Rb^4)$ with the 
symmetric space of $\SO (\wedge^2\Rb^4)\cong \SO(3,3)$. Indeed, the map $A\in \SL(\Rb^4)\mapsto \wedge^2A\in \SO (3,3)$ is an isogeny.

The classical Kummer construction embeds $\wedge^2\Zb^4$ primitively in $\Lambda$ of signature $(3,19)$. The form on $\Lambda$  restricts to $2\delta$ and  the orthogonal complement is spanned by 16 mutually perpendicular vectors having all self product $-2$ 
(these 16  vectors are naturally indexed by $\Fb_2^4$). This yields a totally geodesic  embedding of symmetric spaces 
\[
X_4\cong \Gras^+_3(\wedge^2\Rb^4)\hookrightarrow \Gras^+_3(\Lambda_\Rb).
\]
The right hand side is the Teichm\"uller space for volume one Ricci-flat K3 orbifolds.

We believe that the following is a fundamental problem about the Teichm\"uller space of Ricci-flat metrics on K3 surfaces.

\begin{problem}[{\bf Teichm\"uller problem for Ricci flat manifolds}]
Given two unit volume, Ricci-flat metrics on the K3 manifold $M$ in the same connected component  
in the space of such metrics, find the extremal map  (for e.g.\ the Teichm\"uller potential $\HE$) that realizes their symmetric space distance.
\end{problem}

As a second problem we pose the following.
\begin{problem}
For which closed manifolds $M$, does taking an infimum of the Hilbert-Schmidt expansion 
$\HE$ of maps in an appropriate homotopy class determine a path metric on the space $\Riem (M)_{1}/\Diff(M)^0$?
\end{problem}


\begin{thebibliography}{99} 


\bibitem[BPT05]{BPT05}
A.~Belkhirat, A.~Papadopoulos, M.~Troyanov,
\textit{Thurston's weak metric on the Teichm\" uller space of the torus},
Trans.\ Amer.\ Math.\ Soc.\ \textbf{357} (2005), 3311--3324. \url{https://arxiv.org/pDf/1508.04039}

\bibitem[BNST17]{BNST17}
L.~Borisov, P.~Neff, S.~Sra, C.~Thiel,
\textit{The sum of squared logarithms inequality in arbitrary dimensions,}
Linear Algebra Appl.\  \textbf{528} (2017), 124--146.

\bibitem[Ebin68]{Ebin68}
D.~Ebin,
\textit{The manifold of Riemannian metrics,} in \textsl{Global Analysis, Proc. Sympos.\ Pure Math.}\ \textbf{15}, Amer. Math. Soc., Providence, RI, 1970, 11--40.

%\bibitem[EE69]{EE69}
%C.~Earle, J.~Eells, 
%\textit{A fibre bundle description of Teichm\"uller theory,} J.~Differential Geometry \textbf{3} (1969), 19--43.

\bibitem[Fed69]{Fed69}
H. Federer, Geometric Measure Theory, Springer, 1969. 

\bibitem[FT84]{FT84}
A.~Fischer, A.~Tromba:
\textit{On a purely ``Riemannian'' proof of the structure and dimension of the unramified moduli space of a compact Riemann surface,} Math.\ Ann.  \textbf{267} (1984), 311--345.

\bibitem[GJ21]{GreenfieldJi21}
M.~Greenfield and L.~Ji,
\emph{Metrics and compactifications of Teichmuller spaces of flat tori},
Asian J. Math. \textbf{25} (2021), 477--504. \url{https://arxiv.org/pdf/1903.10655}

\bibitem[Gr82]{Gr82}
M. Gromov, 
\textit{Volume and bounded cohomology,}
Publ.\ Math.\ IHES \textbf{56} (1982), 5--99.

\bibitem[Gu13]{Gu13}
L.~Guth, 
\textit{Contraction of areas vs. topology of mappings,} GAFA, Vol. 23, pp. 1804--1902 (2013).

%\bibitem[Law75]{Law75}
%H.~Blaine Lawson, Jr.,
%\textit{The stable homology of a flat torus,}
%Math.\ Scand.  \textbf{36} (1975), 49--73.

\bibitem[Liu12]{Liu12}
Luofei Liu, 
\textit{Dilatation of homotopy classes and norms of cohomology classes, }
Proc.\ Lond.\ Math.\ Soc.\  \textbf{105} (2012) 447--472.

\bibitem[LV77]{LV77}
J. Luukkainen and J. V\"ais\"al\"a, 
\textit{Elements of Lipschitz topology,}  Ann.\ Acad.\ Sci.\ Fen., 
Ser. A, (1977), 85--122.

\bibitem[MOA11]{MOA11}
A.~W.~Marshall, I.~Olkin, B.~C.~Arnold:
\textsl{Inequalities: Theory of Majorization and Its Applications,} 2nd ed., Springer, 2011.

\bibitem[Mirsky75]{Mirsky75}
L.~Mirsky:
\textit{A trace inequality of John von Neumann,}
Monatsh.\ Math.\ \textbf{79} (1975), 303--306.

\bibitem[Teich39]{Teich39}
O.~Teichm\"uller, \textsl{Extremal quasiconformal mappings and
quadratic differentials,} English trans. by Guillaume Th\'eret, in Handbook on Teichm\"uller Theory, Vol.~V, edited by A. Papadopoulos, IRMA Lect.\ Math.\ Theor.\ Phys.\  \textbf{26}, 321--483. EMS Z\"urich, 2016.

\bibitem[Yam14]{Yam14}
S.~Yamada,
\textit{Local and global aspects of Weil-Petersson geometry,} in \textsl{Handbook of
Teichm\"uller theory,} Vol.~IV, 43--111, edited by A. Papadopoulos, IRMA Lect.\ Math.\ Theor.\ Phys.\ \textbf{19}. EMS Z\"urich, 2014.


\end{thebibliography}
\end{document}